\documentclass[12pt,oneside]{amsart}

\usepackage{amsfonts, amsmath, amssymb, amsthm, amscd, hyperref}

\usepackage{anysize}

\newtheorem{thm}{Theorem}[section]
\newtheorem*{thm*}{Theorem}
\newtheorem{lem}[thm]{Lemma}

\newtheorem{cor}[thm]{Corollary}

\theoremstyle{definition}
\newtheorem{defn}[thm]{Definition}

\theoremstyle{remark}
\newtheorem{rem}[thm]{Remark}

\numberwithin{equation}{section}

\DeclareMathOperator{\conv}{conv}

\DeclareMathOperator{\dist}{dist}

\newcommand{\inte}{\mathop{\rm int}}

\begin{document}

\title{Notes about the Carath\'{e}odory number}

\author{Imre~B\'{a}r\'{a}ny}
\email{barany@renyi.hu}
\address{Imre~B\'{a}r\'{a}ny,
R\'enyi Institute of Mathematics, Hungarian Academy of Sciences, PO Box 127, 1364 Budapest, Hungary;
and Department of Mathematics, University College London, Gower Street, London WC1E 6BT, England}

\author{Roman~Karasev}
\email{r\_n\_karasev@mail.ru}
\address{Roman Karasev, Dept. of Mathematics, Moscow Institute of Physics and Technology, Institutskiy per. 9, Dolgoprudny, Russia 141700; and Laboratory of Discrete and Computational Geometry, Yaroslavl' State University, Sovetskaya st. 14, Yaroslavl', Russia 150000}

\thanks{The work of both authors was supported by ERC Advanced Research Grant No 267195 (DISCONV). The first author acknowledges support from Hungarian National Research Grant No 78439. The second author is supported by the Dynasty Foundation, the President's of Russian Federation grant MK-113.2010.1, the Russian Foundation for Basic Research grants 10-01-00096 and 10-01-00139, the Federal Program ``Scientific and scientific-pedagogical staff of innovative Russia'' 2009--2013, and the Russian government project 11.G34.31.0053.}

\subjclass[2010]{52A35, 52A20} \keywords{Carath\'{e}odory's theorem, Helly's theorem,
Tverberg's theorem}

\begin{abstract}
In this paper we give sufficient conditions for a compactum in $\mathbb R^n$ to have
Carath\'{e}odory number less than $n+1$, generalizing an old result of Fenchel. Then we prove
the corresponding versions of the colorful Carath\'{e}odory theorem and give a Tverberg type
theorem for families of convex compacta.
\end{abstract}

\maketitle

\section{Introduction}

The Carath\'{e}odory theorem~\cite{car1911} asserts that every point $x$ in the convex hull
of a set $X\subset \mathbb R^n$ is in the convex hull of one of its subsets of cardinality at
most $n+1$. In this note we give sufficient conditions for the \emph{Carath\'{e}odory number}
to be less than $n+1$ and prove some related results. In order to simplify the reasoning we
always consider compact subsets of $\mathbb R^n$.

There are results about lowering the Carath\'{e}odory constant: A theorem of
Fenchel~\cite{fen1929,hr1951} asserts that a compactum $X\subset \mathbb R^n$ either has the
Carath\'{e}odory number $\le n$ or can be separated by a hyperplane into two non-empty parts.
By \emph{separated} we mean ``divided by a hyperplane disjoint from $X$ into two non-empty
parts''. In order to state more results we need formal definitions:

\begin{defn}
For a compactum $X\subset\mathbb R^n$ denote by $\conv_{k+1} X$ the sets of points $p\in
\mathbb R^n$ that can be expressed as a convex combination of at most $k+1$ points in $X$. We
denote by $\conv X$ (without subscript) the standard convex hull of $X$.
\end{defn}

\begin{defn}
The \emph{Carath\'{e}odory number} of $X$ is the smallest $k$ such that $\conv X = \conv_k X$.
\end{defn}

\begin{rem}
So Carath\'{e}odory's theorem~\cite{car1911} is equivalent to the equality $\conv X =
\conv_{n+1} X$ when $X \subset \mathbb R^n$. We will give an alternative definition for
$\conv_k X$ in Section~\ref{col-cara-sec} as the $k$-fold join of $X$.
\end{rem}

\begin{defn}
\label{k-conv-def} A compactum $X\subset \mathbb R^n$ is $k$-convex if every linear image of
$X$ to $\mathbb R^k$ is convex.
\end{defn}

We give some examples of $k$-convex sets. What is needed in Fenchel's theorem is
$1$-convexity and every connected set is $1$-convex. The $k$-skeleton of a convex polytope is
$k$-convex (though for such $k$-convex sets most results of this paper are trivial).
In~\cite{brick1961} (see also~\cite[Chapter~II, \S~14]{barv2002}) it is shown that the image
of the sphere under the Veronese map $v_2 : S^{n-1} \to \mathbb R^{n(n+1)/2}$ (with all
degree $2$ monomials as coordinates) is $2$-convex.

In~\cite[Corollary~1]{hr1951} the following remarkable result is proved:

\begin{thm}[Hanner--R\aa dstr\"{o}m, 1951]
\label{hann-rad}
If $X$ is a union of at most $n$ compacta $X_1,\ldots, X_n$ in $\mathbb R^n$
and each $X_i$ is $1$-convex then $\conv_n X = \conv X$.
\end{thm}

It is also known~\cite{ks1966,barv2002} that a \emph{convex curve} in $\mathbb R^n$ (that is
a curve with no $n+1$ points in a single affine hyperplane) has Carath\'{e}odory number at
most $\lfloor\frac{n+2}{2}\rfloor$. It would be interesting to obtain some nontrivial bounds
for the Carath\'{e}odory number of the orbit $Gx$ of a point $x$ in a representation $V$ of a
compact Lie group $G$ in terms of $\dim V$, $\dim G$ (or the rank of $G$). The latter
question is mentioned in~\cite[Question~3]{sss2009} and would be useful in results like those
in~\cite{mil1988}.

In Section~\ref{k-conv-cara-sec} of this paper we show that the Carath\'{e}odory number is
$\le k+1$ for $(n-k)$-convex sets. In Section~\ref{col-cara-sec} we prove the corresponding
analogue of the colorful Carath\'{e}odory theorem, and in Section~\ref{tverberg-sec} we give
a related Tverberg-type result.

\section{The Carath\'{e}odory number and $k$-convexity}
\label{k-conv-cara-sec}

We are going to give a natural generalization of the reasoning in~\cite{hr1951}:

\begin{thm}
\label{k-sep} Suppose $X_1, \ldots, X_{n-k}$ are compacta in $\mathbb R^n$ and $p$ does not
belong to $\conv_{k+1} X_i$ for any $i$. Then there exists an affine $k$-plane $L\ni p$ that
has empty intersection with any $X_i$.
\end{thm}

\begin{rem}
If we replace $\conv_{k+1} X_i$ by the honest convex hull $\conv X_i$ then the result is
simply deduced by induction from the Hahn--Banach theorem.
\end{rem}

\begin{rem}
In~\cite{kos1962} a somewhat related result was proved: For a compactum $X\subset \mathbb R^n$ and a point $p\not\in X$ there exists an affine $k$-plane $L$ (for a prescribed $k<n$) such that the intersection $L\cap K$ is not acyclic modulo $2$. Here \emph{acyclic} means having the \v{C}ech cohomology of a point.
\end{rem}

The proof of Theorem~\ref{k-sep} is given in Section~\ref{k-sep-sec}. Now we deduce the following generalization of Fenchel's theorem~\cite{fen1929}:

\begin{cor}
\label{k-conv-cara}
If a compactum $X\subset\mathbb R^n$ is $(n-k)$-convex then $\conv_{k+1} X = \conv X$.
\end{cor}

\begin{proof}
Assume the contrary and let $p\in \conv X\setminus \conv_{k+1} X$. Applying
Theorem~\ref{k-sep} to the family $\underbrace{X,\ldots, X}_{n-k}$ we find a $k$-dimensional
$L\ni p$ disjoint from $X$. Now project $X$ along $L$ with $\pi : \mathbb R^n\to \mathbb
R^{n-k}$. Since $X$ is $(n-k)$-convex $\pi(L)$ must be separated from $X$ by a hyperplane.
Hence $L$ is separated from $X$ by a hyperplane and therefore $p$ cannot be in $\conv X$.
\end{proof}

\begin{rem}
In the above lemma and its proof we could consider $n-k$ different $(n-k)$-convex compacta
$X_1,\ldots, X_{n-k}$ and by the same reasoning obtain the following conclusion:
$$
\bigcup_{i=1}^{n-k} \conv_{k+1} X_i = \bigcup_{i=1}^{n-k} \conv X_i.
$$
But this result trivially follows from Corollary~\ref{k-conv-cara} by taking the union.
\end{rem}

\begin{rem}
For the image $v_2(S^{n-1})$ of the Veronese map the Carath\'{e}odory constant is roughly of
order $n$, see~\cite[Chapter~II, \S~14, Theorem~14.3]{barv2002}. Hence
Corollary~\ref{k-conv-cara} is not optimal for this set.
\end{rem}

\section{Proof of Theorem~\ref{k-sep}}
\label{k-sep-sec}

Let us replace $X_i$ by a smooth nonnegative function $\rho_i$ such that $\rho_i>0$ on $X_i$
and $\rho_i = 0$ outside some $\varepsilon$-neighborhood of $X_i$. Let $p$ be the origin.

Assume the contrary: for any $k$-dimensional linear subspace $L\subset \mathbb R^n$ some
intersection $L\cap X_i$ is nonempty. The space of all possible $L$ is the Grassmann manifold
$G_n^k$. Denote by $D_i$ the open subset of $G_n^k$ consisting of $L\in G_n^k$ such that
$\int_L \rho_i > 0$. Note that $0$ cannot lie in the convex hull $\conv (L\cap X_i)$ because
in this case by the ordinary Carath\'{e}odory theorem $0$ would be in $\conv_{k+1} (L\cap
X_i) \subseteq \conv_{k+1} X_i$, contradicting the hypothesis. Hence (if we choose small
enough $\varepsilon>0$) the ``momentum'' integral
$$
m_i (L) = \int_L \rho_i x\; dx
$$
never coincides with $0$ over $D_i$. Obviously $m_i(L)$ is a continuous section of the
canonical vector bundle $\gamma : E(\gamma) \to G_n^k$, which is nonzero over $D_i$. Now we
apply:

\begin{lem}
\label{eulergrass} Any $n-k$ sections of $\gamma : E(\gamma) \to G_n^k$ have a common zero
because of the nonzero Euler class $e(\gamma)^{n-k}$.
\end{lem}

This lemma is a folklore fact, see~\cite{dol1987,zivvre1990} for example. Applying this lemma
to the sections $m_i$ we obtain that the sets $D_i$ do not cover the entire $G_n^k$. Hence
some $L\in G_n^k$ has an empty intersection with every $X_i$. \hfill $\Box$

\begin{rem}
In the proof of Theorem~\ref{hann-rad} in~\cite{hr1951} instead of finding a zero of a
section of a vector bundle over $\mathbb RP^{n-1}$ some analogue of the Brouwer fixed point
theorem is used for a convex subset of the sphere $S^{n-1}$.
\end{rem}

\section{The colorful Carath\'{e}odory number}
\label{col-cara-sec}

Let us introduce some notation and restate the colorful Carath\'{e}odory theorem~\cite{bar1982}.

\begin{defn}
Denote $A*B$ the \emph{geometric join} of two sets $A,B\in \mathbb R^n$, which is
$$
\{ta + (1-t)b : a\in A, b\in B,\ \text{and}\ t\in [0, 1]\}.
$$
This is actually the alternative definition of $\conv_k X$ as $\underbrace{X* \dots* X}_k$.
\end{defn}

\begin{thm}[B\'{a}r\'{a}ny, 1982]
\label{col-cara} If $X_1,\ldots, X_{n+1}\subset \mathbb R^n$ are compacta and $0\in \conv
X_i$ for every $i$ then $0\in X_1*X_2*\dots * X_{n+1}$.
\end{thm}


It is possible to reduce the Carath\'{e}odory number $n+1$ assuming the $(n-k)$-convexity of
$X_i$, thus generalizing Corollary~\ref{k-conv-cara}:

\begin{thm}
\label{k-conv-color} Let $0\le k \le n$. If $X_1,\ldots, X_{k+1}\subset \mathbb R^n$ are
$(n-k)$-convex compacta and $0\in \conv X_i$ for every $i$ then $0\in X_1*X_2*\dots *
X_{k+1}$.
\end{thm}

\begin{proof}
We use the classical scheme~\cite{bar1982} along with the degree reasoning used
in~\cite{dhst2006,barmat2007,stth2008,abbfm2009} in the proof of different generalizations of
the colorful Carath\'{e}odory theorem.

Consider the case $k = n-1$ first. In this case we have $n$ sets and $1$-convexity. Let
$x_1,\ldots, x_n$ be the system of representatives of $X_1,\ldots, X_n$ such that the
distance from $S = \conv\{x_1, \ldots, x_n\}$ to the origin is minimal. If this distance is
zero then we are done. Otherwise assume that $z\in S$ minimizes the distance.

Let $z= t_1x_1 + \dots +t_n x_n$, a convex combination of the $x_i$s. If $t_i=0$ then we
observe that $0\in \conv X_i$ and we can replace $x_i$ by another $x_i'$ so that new simplex
$S'= \conv\{x_1, \ldots, x_{i-1}, x'_i, x_{i+1}, \ldots, x_n\}$ is closer to the origin than
$S$. So we may assume that all the coefficients $t_i$ are positive and $z$ is in the relative
interior of $S$. This also implies that $S$ is $(n-1)$-dimensional, i.e., there is a unique
hyperplane containing $S$.

Consider the hyperplane $h\ni 0$ parallel to $S$. Applying the definition of $1$-convexity to
the projection along $h$ we obtain that there exists a system of representatives $y_i\in
X_i\cap h$. The set
$$
f(B) = \{x_1, y_1\}*\{x_2, y_2\}*\dots * \{x_n, y_n\}
$$
is a piece-wise linear image of the boundary of a crosspolytope, which we denote by $B$. Note
that for every facet $F$ of $B$, the vertices of the simplex $f(F)$ form a system of
representatives for $\{X_1, \ldots, X_n\}$. In particular, $S=f(F)$ for some facet $F$ of
$B$. The line $\ell$ through the origin and $z$ intersects the simplex $S=f(F)$ transversally
and so it must intersect some other $f(F')$ (where $F'\ne F$ is a facet of $B$) because of
the parity of the intersection index. The intersection $\ell\cap f(F')$ is on the segment
$[0,z]$ and cannot coincide with $z$. Therefore $f(F')$ is closer to the origin than $S$.
This is a contradiction with the choice of $S$. Thus the case $k=n-1$ is done.

The case $k=0$ of this theorem is trivial by definition, the case $k=n$ corresponds to the
colorful Carath\'eodory theorem. Now let $0<k <n-1$. Consider again a system of
representatives $x_1, \ldots, x_{k+1}$ minimizing the distance $\dist(0, \conv\{x_1, \ldots,
x_{k+1}\})$. Put $S=\conv \{x_1, \ldots, x_{k+1}\}$. As above the closest to the origin point
$z\in S$ must lie in the relative interior of $S$ if $z\neq 0$.

Let $L\subset\mathbb R^n$ be the $k$-dimensional linear subspace parallel to $S$. As in the
first proof using $(n-k)$-convexity we select $y_i\in L\cap X_i$. Then we map naturally the
boundary $B$ of a $(k+1)$-dimensional crosspolytope to the geometric join
$$
f(B) = \{x_1, y_1\}*\{x_2, y_2\}*\dots * \{x_{k+1}, y_{k+1}\}.
$$
Note that $f(B)$ is contained in the $(k+1)$-dimensional linear span of $S$ and $L$, so by
the parity argument as above the image under $f$ of some face of $B$ must be closer to the
origin than $S$.
\end{proof}

\begin{rem}
In this proof in the case $k < n-1$ we can choose some $(k+1)$-dimensional subspace $M\subset
\mathbb R^n$ and a system of representatives $\{x_1, \ldots, x_{k+1}\}$ for $M\cap X_1,
\ldots, M\cap X_{k+1}$. Then we can make the steps reducing $\dist(0, \conv \{x_1, \ldots,
x_{k+1}\})$ so that the system of representatives always remains in $M$.
\end{rem}

\section{A topological approach to Theorem~\ref{k-conv-color}}

Theorem~\ref{k-conv-color} can also be deduced from the following lemma:

\begin{lem}
\label{join-in-bundle} Let $\xi : E(\xi) \to X$ be a $k$-dimensional vector bundle over a
compact metric space $X$. Let $Y_1, \ldots, Y_{k+1}$ be closed subspaces of $E(\xi)$ such
that for every $i$ the projection $\xi|_{Y_i} : Y_i\to X$ is surjective. If $e(\xi)\neq 0$
then for some fiber $V = \xi^{-1}(x)$ the geometric join
$$
(Y_1\cap V)*\dots *(Y_{k+1}\cap V)
$$
contains $0\in V$.
\end{lem}

\begin{rem}
The Euler class here may be considered in integral cohomology or in the cohomology mod $2$.
The proof passes in both cases so we omit the coefficients from the notation.
\end{rem}

\begin{proof}[Reduction of Theorem~\ref{k-conv-color} to Lemma~\ref{join-in-bundle} for $k<n$]
Take a linear subspace $M\subseteq \mathbb R^n$ of dimension $k+1$. For every $k$-dimensional
linear subspace $L\subset M$ all the intersections $L\cap X_i$ are nonempty. All such $L$
constitute the canonical bundle $\gamma$ over $G_{k+1}^k = \mathbb RP^k$ with nonzero Euler
class by Lemma~\ref{eulergrass}. For any fixed $i$ the union of sets $L\cap X_i$ constitute a
closed subset of $E(\gamma)$ that we denote by $Y_i$. By Lemma~\ref{join-in-bundle} for some
$L$ the join
$$
(Y_1\cap L)*\dots *(Y_{k+1}\cap L) = (X_1\cap L)*\dots *(X_{k+1}\cap L)
$$
must contain the origin.
\end{proof}

\begin{proof}[Now we prove Lemma~\ref{join-in-bundle}]
The proof has much in common with the results of~\cite{kos1962}. The main idea is that fiberwise acyclic (up to some dimension) subsets of the total space of a vector bundle behave like sections of that vector bundle.

Let $Y=Y_1*_X*\dots *_X Y_{k+1}$ be the \emph{abstract fiberwise join} over $X$, that is the
set of all formal convex combinations
$$
t_1 y_1 + t_2 y_2 +\dots + t_{k+1}y_{k+1},
$$
where $t_i$ are nonnegative reals with unit sum and $y_i\in Y_i$ are points such that
$$
\xi(y_1) = \dots = \xi(y_{k+1}).
$$
Denote the natural projection $\eta : Y \to X$. Any formal convex combination $y\in Y$
defines a corresponding ``geometric'' convex combination $f(y)$ in the fiber
$\xi^{-1}(\eta(y))$ depending continuously on $y$. It is easy to check that $f(y)$ can be
considered as a section of the pullback vector bundle $\eta^*(\xi)$ over $Y$.

For any point $x\in X$ its preimage under $\eta$ is a join of $(k+1)$ nonempty sets
$$
(Y_1\cap \xi^{-1}(x)) * \dots * (Y_{k+1}\cap \xi^{-1}(x))
$$
and therefore $\eta^{-1}(x)$ is $(k-1)$-connected. Hence the Leray spectral sequence for the \v{C}ech cohomology $H^*(Y)$ with $E_2^{*,*} = H^*(X; \mathcal H^*(\eta^{-1}(x)))$ (the coefficient sheaf is the direct image of the homology of the total space) has empty rows number $1,\ldots, k-1$ and its differentials cannot kill the image of $e(\xi)$ in $E_r^{k, 0}$.
Hence $\eta^* (e(\xi)) = e(\eta^*(\xi))$ remains nonzero over $Y$ and by the standard
property of the Euler class for some $y\in Y$ the section $f(y)$ must be zero.
\end{proof}

\begin{rem}
In this proof we essentially use the inequality $k<n$. So the colorful Carath\'{e}odory
theorem is \emph{not} a consequence of Lemma~\ref{join-in-bundle}, at least in our present
state of knowledge.
\end{rem}

The subsets $Y_i$ in Lemma~\ref{join-in-bundle} can be considered as \emph{set-valued}
sections. The same technique proves the following:

\begin{thm}
\label{color-fp} Let $B$ be an $n$-dimensional ball and $f_i : B \to 2^B\setminus\emptyset$
for $i=1,\ldots, n+1$ be set-valued maps with closed graphs (in $B\times B$). Then for some
$x\in B$ the inclusion holds:
$$
x\in f_1(x)*\dots *f_{n+1}(x).
$$
\end{thm}

\begin{proof}
We may assume that all sets $f_i(x)$ are in the interior of $B$, because the general case is
reduced to this one by composing $f_i$ with a homothety with scale $1-\varepsilon$ and going
to the limit $\varepsilon\to+0$.

It is known~\cite{kar2008} that for a single-valued map $f : B\to \inte B$ (considered as a
section of the trivial bundle $B\times \mathbb R^n\to B$) a fixed point ($x=f(x)$) is
guaranteed by the relative Euler class $e(f(x)-x)\in H^n(B, \partial B)$. Then the proof
proceeds as in Lemma~\ref{join-in-bundle} by lifting $e(f(x)-x)$ to the abstract fiberwise
join of graphs of $f_i$ over the pair $(B, \partial B)$ and using the properties of the
relative Euler class of a section.
\end{proof}

\begin{cor}
Suppose $X_1,\ldots, X_{n+1}$ are compacta in $\mathbb R^n$ and $\rho$ is a continuous metric
on $\mathbb R^n$. For any $x\in\mathbb R^n$ denote by $f_i(x)$ the set of farthest point from
$x$ in $X_i$ (in the metric $\rho$. Then for some $x\in \mathbb R^n$ we have
$$
x\in f_1(x)*\dots *f_{n+1}(x).
$$
\end{cor}

\begin{rem}
If we denote by $f_i(x)$ the \emph{closets} points in $X_i$ then this assertion becomes
almost trivial without using any topology.
\end{rem}

\section{The Carath\'{e}odory number and the Tverberg property}
\label{tverberg-sec}

Tverberg's classical theorem~\cite{tver1966} says the following:

\begin{thm}[Tverberg, 1966]
Every set of $(n+1)(r-1)+1$ points in $\mathbb R^n$ can be partitioned into $r$ parts
$X_1,\ldots, X_r$ so that the convex hulls $\conv X_i$ have a common point.
\end{thm}

From the general position considerations it is clear that the number $(n+1)(r-1)+1$ cannot be
decreased. But we are going to decrease it after replacing a finite point set by a family of
convex compacta. Let us define the Carath\'{e}odory number for such families:

\begin{defn}
Suppose $\mathcal F$ is a family of convex compacta in $\mathbb R^n$. The
\emph{Carath\'eodory number} of $\mathcal F$ is the least $\kappa$ such that for any
subfamily $\mathcal G\subseteq \mathcal F$
$$
\conv \bigcup\mathcal G = \bigcup_{\mathcal H\subseteq\mathcal G,\ |\mathcal H| \le \kappa} \conv\bigcup\mathcal H.
$$
We denote the Carath\'{e}odory number of $\mathcal F$ by $\kappa(\mathcal F)$.
\end{defn}

Again, from the Carath\'{e}odory theorem~\cite{car1911} it follows that $\kappa(\mathcal
F)\le n+1$. Another observation is that Corollary~\ref{k-conv-cara} guarantees that
$\kappa(\mathcal F)\le k+1$ if the union of every subfamily $\mathcal G\subseteq \mathcal F$
is $(n-k)$-convex.

Now we state the analogue of Tverberg's theorem:

\begin{thm}
\label{tverberg-families}
Suppose $\mathcal F$ is a family of convex compacta in $\mathbb R^n$, $r$ is a positive integer, and
$$
|\mathcal F| \ge r\kappa(\mathcal F) + 1.
$$
Then $\mathcal F$ can be partitioned into $r$ subfamilies $\mathcal F_1,\ldots,\mathcal F_r$ so that
$$
\bigcap_{i=1}^r\conv\bigcup\mathcal F_i\neq\emptyset.
$$
\end{thm}

\begin{rem}
Note the following: If $\kappa(\mathcal F) = n+1$ then taking a system of representatives for
$\mathcal F$ and applying the Tverberg theorem we obtain a weaker condition: $|\mathcal F|
\ge (r-1)(n+1) + 1$.
\end{rem}

\begin{rem}
This theorem originated in discussions with Andreas~Holmsen, who established the same result
in the special case $n = 2$, $\kappa(\mathcal F) = 2$, and with $|\mathcal F| \ge 2r$
(\emph{not} $2r+1$).
\end{rem}

\begin{proof}[Proof of Theorem~\ref{tverberg-families}]
We again use a minimization argument, combined with Sarkaria's tensor trick~\cite{sark1992}.
Let $|\mathcal F| = m$, $\kappa = \kappa(\mathcal F)$, and
$$
\mathcal F= \{C_1, C_2, \ldots, C_m\}.
$$
Put the space $\mathbb R^n$ to $A=\mathbb R^{n+1}$ as a hyperplane given by the equation
$x_{n+1}=1$. Consider a set $S$ of vertices of a regular simplex in some $(r-1)$-dimensional
space $V$ and assume that $S$ is centered at the origin.

Now define the subsets of $V\otimes A$ by
$$
X_i = S\otimes C_i,
$$
and consider a system of representatives $(x_1, x_2,\ldots, x_m)$ for the family of sets
$\mathcal G = \{X_1,X_2,\ldots, X_m\}$. Such a system gives rise to a partition $\{P_s: s\in
S\}$ of $\{1,\dots,m\}$ the following way. For $s \in S$ define
$$
P_s=\{i\in \{1,\dots,m\}:x_i = s\otimes c_i,\ \text{for some}\ c_i\in C_i\}.
$$

One form of Sarkaria's trick, Lemma 2 in \cite{abbfm2009} says that $0 \in \conv
\{x_1,\dots,x_m\}$ if and only if $\bigcap_{s\in S}\conv\{c_i: i\in P_s\}\ne \emptyset$.
Based on this we choose a system of representatives $(x_1,\dots,x_m)$ of $\mathcal G$ so that
the distance between $0$ and $\conv\{x_1, x_2, \ldots, x_m\}$ is minimal. If this distance is
zero then the required partition of $\mathcal F$ is given by the sets $\{C_i \in \mathcal F:
i\in P_s\}$, $s\in S$.

Assume that the minimal distance is not zero. Then it is attained on some convex combination
$$
x_0 = \alpha_1x_1 + \alpha_2x_2+\dots+\alpha_mx_m.
$$
We {\bf claim} that $\alpha_i>0$ for all $i \in \{1,\dots,m\}$. Assume, for instance that
$\alpha_1=0$, and $x_1=s\otimes c_1$ for some $c_1\in C_1$ and $s \in S$. Now $x_1$ can be
replaced by $t\otimes c_1$ for any $t \in S$ as such a change does not influence $x_0$. The
distance minimality condition implies that all the points $t\otimes c_1$ are separated from
the origin by a hyperplane in $V\otimes A$, which is the support hyperplane for the ball,
centered at the origin and touching $\conv\{x_1,\ldots, x_m\}$. Obviously
$$
\sum_{t\in S} t\otimes c_i = 0,
$$
so the points $t\otimes c_i$, $t \in S$ are not separated from the origin. This contradiction
completes the proof of the claim.

The above convex combination representing $x_0$ can be written as
$$
x_0 = \sum_{s\in S} s\otimes\left(\sum_{i\in P_s} \alpha_ic_i\right).
$$
Assume first that no $P_s$ is the emptyset. Define $c(s)=\sum_{i\in P_s} \alpha_ic_i$ and
$\alpha(s)=\sum_{i\in P_s} \alpha_i >0$. Then $c(s)/\alpha(s)$ is a convex combination of
elements $c_i\in C_i, i \in P_s$. Thus $c(s)/\alpha(s) \in \bigcup_{i\in P_s}C_i$. According
to the definition of the Carath\'{e}odory number, there is a subset $P_s'\subset P_s$, of
size at most $\kappa$, such that $c(s)/\alpha(s) \in \bigcup_{i\in P_s'}C_i$ for every $s \in
S$. This means that there are $c_i' \in C_i$ for all $i \in P_s'$ such that $c(s)/\alpha(s)
\in \conv \{c_i': i \in P_s'\}$, in other words, $c(s)=\sum_{i \in P_s'}\alpha_i'c_i'$ with
positive $\alpha_i'$ satisfying $\sum_{i \in P_s'}\alpha_i'=\alpha(s)$. Thus
$$
x_0 = \sum_{s\in S} s\otimes\left(\sum_{i\in P_s'} \alpha_i'c_i\right).
$$
In this case the minimum distance is attained on the convex hull of no more that $r\kappa$
elements as each $|P_s'|\le \kappa$. But $m>r\kappa$ contradicting the claim.

Finally we have deal with the (easy) case when some $P_s=\emptyset$. The above argument
works, with no change at all, for the non-empty $P_s$ implying that $x_0$ can be written as
the convex combination of at most $(r-1)\kappa$ elements. Again $m>(r-1)\kappa$ and the same
contradiction finishes the proof.
\end{proof}

\vskip 1cm {\bf Acknowledgment.} We thank Peter~Landweber who has drawn our attention to those old results by Fenchel, Hanner, and R\aa dstr\"{o}m and Alexander~Barvinok for discussions and examples of $k$-convexity.

\end{document}